\documentclass[11pt, a4paper]{article}
\usepackage{hyperref}
\usepackage{latexsym}
\usepackage{amsmath, epsfig}
\usepackage{color}
\usepackage{graphicx, float, enumerate}
\usepackage{amssymb,amsfonts,amsthm,amsmath,graphicx,url}
\usepackage{authblk}

\input amssym.def
\input amssym.tex

\def\qed{\ifmmode\square\else\nolinebreak\hfill$\square$\fi\par\vskip12pt}

\def\Aut{{\rm Aut}} 
\def\a{\alpha} 
\def\Cay{{\rm Cay}} 
\def\PSL{{\rm PSL}}
\def\l{\langle} 
\def\r{\rangle} 
\def\D{\Delta} 
\def\Inn{{\rm Inn}} 
\def\AGL{{\rm AGL}}
\def\Symp{{\rm Sp}} 
\def\L{{\rm L}} 
\def\M{{\rm M}} 
\def\Out{{\rm Out}}
\def\Sym{{\rm Sym}}
\def\soc{{\rm soc}}

\newtheorem{theorem}{Theorem}[section]
\newtheorem{lemma}[theorem]{Lemma}

\newtheorem{proposition}[theorem]{Proposition}

\long\def\delete#1{}

\usepackage{xcolor}
\usepackage[normalem]{ulem}

\begin{document}

\title{Classification of tetravalent $2$-transitive non-normal Cayley graphs of finite simple groups}
		
\author[a]{Xin Gui Fang\thanks{E-mail: \texttt{xgfang@math.pku.edu.cn}}}
\author[a]{Jie Wang\thanks{E-mail: \texttt{wangj@pku.edu.cn}}}
\author[b]{Sanming Zhou\thanks{E-mail: \texttt{sanming@unimelb.edu.au}}}
\affil[a]{{\small LAMA and School of Mathematical Sciences, Peking University, Beijing 100871, P. R. China}}
\affil[b]{{\small School of Mathematics and Statistics, The University of Melbourne, Parkville, VIC 3010, Australia}}
 
\date{}

\openup 0.6\jot

\maketitle

\begin{abstract}
A graph $\Gamma$ is called $(G, s)$-arc-transitive if $G \le \Aut(\Gamma)$ is transitive on the set of vertices of $\Gamma$ and the set of $s$-arcs of $\Gamma$, where for an integer $s \ge 1$ an $s$-arc of $\Gamma$ is a sequence of $s+1$ vertices $(v_0,v_1,\ldots,v_s)$ of $\Gamma$ such that $v_{i-1}$ and $v_i$ are adjacent for
$1 \le i \le s$ and $v_{i-1}\ne v_{i+1}$ for $1 \le i \le s-1$. $\Gamma$ is called 2-transitive if it is $(\Aut(\Gamma), 2)$-arc-transitive but not $(\Aut(\Gamma), 3)$-arc-transitive. A Cayley graph $\Gamma$ of a group $G$ is called normal if $G$ is normal in $\Aut(\Gamma)$ and non-normal otherwise. It was proved by X. G. Fang, C. H. Li and M. Y. Xu that if $\Gamma$ is a tetravalent 2-transitive Cayley graph of a finite simple group $G$, then either $\Gamma$ is normal or $G$ is one of the groups $\PSL_2(11)$, $\M_{11}$, $\M_{23}$ and $A_{11}$. However, it was unknown whether $\Gamma$ is normal when $G$ is one of these four groups. In the present paper we answer this question by proving that among these four groups only $\M_{11}$ produces connected tetravalent 2-transitive non-normal Cayley graphs. We prove further that there are exactly two such graphs which are non-isomorphic and both determined in the paper. As a consequence, the automorphism group of any connected tetravalent 2-transitive Cayley graph of any finite simple group is determined. 
 
\medskip
\textbf{Keywords:} Cayley graph; $s$-arc-transitive graph; $2$-transitive graph; finite simple group

\textbf{AMS 2010 mathematics subject classiﬁcation:} 05C25, 20B25
\end{abstract}

\section{Introduction}

All groups considered in the paper are finite, and all graphs considered are finite, simple and undirected. Given a group $G$ and a subset $S$ of $G$ such that $1_G\notin
S$ and $S = S^{-1} := \{x^{-1}: x \in S\}$, the {\em Cayley graph} of $G$
relative to $S$ is defined to be the graph $\Gamma = \Cay(G,S)$ with vertex set $V\Gamma =G$ and edge set $E\Gamma =\{\{x,y\}\mid yx^{-1}\in S\}$. 
It is readily seen that $\Gamma$ has valency $|S|$. It is also easy to see that $\Gamma$ is connected if and only if $S$ is a generating set of $G$. In general, $\Gamma$ has exactly $|G:\langle S\rangle|$ connected components, each of which is isomorphic to $\Cay(\l S\r,S)$, where $\l S\r$ is the subgroup of $G$ generated by $S$. So we may focus on the connected case when dealing with Cayley graphs. Denote by $G_R$ the right regular representation of $G$. Define
$$A(G,S) := \{\,x\in\Aut(G)\mid S^x=S\,\}.$$
Then $A(G,S)$ is a subgroup of $\Aut(G)$ acting naturally on $G$. It is not difficult to see that $\Gamma=\Cay(G,S)$ admits $G_RA(G,S)$ as a subgroup of its full automorphism group $\Aut(\Gamma)$. It is well known (see \cite{Godsil, Xu}) that ${\rm
N}_{\Aut(\Gamma)}(G_R)=G_RA(G,S)$. Since $G_R\cong G$, we may use $G$ in place of $G_R$, so that $G_RA(G,S)$ is written as $G.A(G,S)$. $\Gamma$ is called a {\it normal} Cayley graph if $G$ is normal in $\Aut(\Gamma)$, that is, $\Aut(\Gamma)=G.A(G,S)$.

A fundamental problem in studying the structure of a graph
is to determine its full automorphism group. This is, in general,
quite difficult. However, for a connected Cayley graph $\Gamma =\Cay(G,S)$
of valency $d$, if $\Gamma$ is normal, then we know that its automorphism group is given by $\Aut(\Gamma)=G.A(G,S)$. Moreover, the subgroup $A(G,S)$ of $\Aut(G)$ acts faithfully on the {\em neighbourhood} $\Gamma(\a)$ of $\a\in V\Gamma$, where $\Gamma(\a)$ is defined as the set of vertices of $\Gamma$ adjacent to $\a$ in $\Gamma$. Hence $A(G,S)$ is
isomorphic to a subgroup of the symmetric group $S_d$ of degree $d$. In other words, if
$\Gamma$ is a normal Cayley graph, then the structure of $\Aut(\Gamma)$ is well
understood. In contrast, it is more challenging to determine the automorphism groups of non-normal Cayley graphs. As such non-normal Cayley graphs have attracted considerable attention in recent years.

Given an integer $s \ge 1$, an {\em s-arc} of a graph $\Gamma$ is a sequence $(v_0,v_1,\ldots,v_s)$ of $s+1$ vertices of $\Gamma$ such that $\{v_{i-1},v_i\}\in E\Gamma$ for
$i=1, 2, \ldots, s$ and $v_{i-1}\ne v_{i+1}$ for $i=1,2,\ldots,s-1$.
A graph $\Gamma$ is called {\it $(G,s)$-arc-transitive} if $G$ is a 
subgroup of $\Aut(\Gamma)$ that is transitive on $V\Gamma$ and transitive on the set
of $s$-arcs of $\Gamma$. A $(G, s)$-arc-transitive graph is called {\em $(G,s)$-transitive} if it is not $(G,s+1)$-arc-transitive. In particular, $\Gamma$ is called 
{\em $s$-arc-transitive} if it is $(\Aut(\Gamma), s)$-arc-transitive, and {\em $s$-transitive} if it is $(\Aut(\Gamma),s)$-transitive. A 1-arc-transitive graph is also called an {\em arc-transitive} or {\em symmetric} graph.
  
For any integer $s\geq 1$, a complete classification of cubic 
$s$-transitive non-normal Cayley graphs of finite simple groups was
obtained by S. J. Xu, M. Y. Xu and the first two authors of
the present paper (see \cite{XFWX1, XFWX2}). In the tetravalent case, C. H. Li, M. Y. Xu and the first author of the present paper proved (\cite[Theorem 1.1]{FLX}) that, if $\Gamma$ is a tetravalent 2-transitive Cayley graph of a finite simple group $G$, then either $\Gamma$ is normal or $G$ is one of the following groups: $\PSL_2(11)$ (two-dimensional projective special linear group over $\mathbb{F}_{11}$); $\M_{11}$ (Mathieu group of degree $11$); $\M_{23}$ (Mathieu group of degree $23$); $A_{11}$ (alternating group  of degree $11$). However, for a long time it was unknown whether $\Gamma$ is normal when $G$ is one of these four groups. In this paper we settle these unsolved cases and classify all connected tetravalent 2-transitive non-normal Cayley graphs of finite simple groups. As a consequence, the automorphism group of any connected tetravalent 2-transitive Cayley graph of any finite simple group is determined. 

The main result of this paper is as follows, where the graphs $\Gamma(\D_1)$ and $\Gamma(\D_2)$ involved will be defined in (\ref{eqn:3}) in Section 3.

\begin{theorem}\label{them1}
Let $G$ be a finite nonabelian simple group and $\Gamma=\Cay(G,S)$ a
connected tetravalent $2$-transitive Cayley graph of $G$. Then one of the following occurs:
\begin{enumerate}[\rm (a)]
	\item $\Gamma$ is normal, and $\Aut(\Gamma) = G.A_4$ or $G.S_4$;
	\item $G=\M_{11}$, $\Aut(\Gamma)=\Aut(\M_{12})=\M_{12}{:}2$, $\Aut(\Gamma)_{\alpha}\cong S_4$ for $\alpha\in V\Gamma$, $\Gamma$ is non-normal, $\Gamma \cong \Gamma(\D_1)$ or $\Gamma(\D_2)$, and $\Gamma(\D_1)$ and $\Gamma(\D_2)$ are not isomorphic.
\end{enumerate}
\end{theorem}

In the next section we introduce notation and give a few preliminary results. In Section 3 we determine all tetravalent 2-transitive non-normal Cayley graphs of finite simple groups by analyzing the four groups above. In Section 4 we settle the isomorphism problem and thus complete the proof of Theorem \ref{them1}. As we will see shortly, even in the four innocent-looking cases above, considerable analysis and computation will be needed in order to establish Theorem \ref{them1}. We will also use \cite[Theorem 1.1]{FLX} in our proof of Theorem \ref{them1}.

\section{Preliminaries}

A permutation group $G$ acting on a set $\Omega$
is said to be {\em quasiprimitive} if each of its nontrivial
normal subgroups is transitive on $\Omega$. 
The {\em socle} of a group $G$, denoted by $\soc(G)$, is the product of all
minimal normal subgroups of $G$. In particular, $G$ is said to be {\it
almost simple} if $\soc(G)$ is a nonabelian simple group. Given a graph
$\Gamma$ and a group $K\leq\Aut(\Gamma)$, the {\em quotient graph} $\Gamma_K$ of $\Gamma$
relative to $K$ is defined as the graph with vertices the $K$-orbits on $V\Gamma$, 
such that two $K$-orbits, say, $X$ and $Y$, are adjacent in
$\Gamma_K$ if and only if there is an edge of $\Gamma$ with one end-vertex
in $X$ and the other end-vertex in $Y$.

The following lemma determines the vertex stabilizers for connected tetravalent 2-transitive graphs (see~\cite[Theorem 4]{Potocnik} or~\cite[Proposition 2.2]{GFL}).

\begin{lemma}\label{PGFL}
Let $\Gamma$ be a connected tetravalent $2$-transitive graph. Then the
vertex stabilizer of $\Gamma$ is $A_4$ or $S_4$.
\end{lemma}

The next lemma describes possible structure of the full
automorphism group of a connected Cayley graph of a finite simple
group.

\begin{lemma}{\rm (\cite[Theorem 1.1]{FPW})}
\label{fpw}
Let $G$ be a finite nonabelian simple group and $\Gamma=\Cay(G,S)$ a
connected Cayley graph of $G$. Let $M$ be a subgroup of $\Aut(\Gamma)$
containing $G.A(G,S)$. Then either $M=G.A(G,S)$ or one of the
following holds:
\begin{enumerate}[\rm (a)]
\item $M$ is almost simple, and $\soc(M)$ contains $G$ as
a proper subgroup and is transitive on $V\Gamma$;

\item $G\cdot\Inn(G)\leq M=G\cdot A(G,S)\cdot2$ and
$S$ is a self-inverse union of $G$-conjugacy classes;

\item $M$ is not quasiprimitive and there is a maximal
intransitive normal subgroup $H$ of $M$ such that one of the
following holds:
\begin{enumerate}[\rm (i)]
\item $M/H$ is almost simple, and $\soc(M/H)$ contains
$GH/H\cong G$ and is transitive on $V\Gamma_H$;

\item $M/H=\AGL_3(2)$, $G=\L_2(7)$, and $\Gamma_H\cong K_8$;

\item $\soc(M/H)\cong T\times T$, and $GH/H\cong G$ is a
diagonal subgroup of $\soc(M/H)$, where $T$ and $G$ are given in Table 1.
\end{enumerate}
\end{enumerate}

Moreover, there are examples of connected Cayley graphs of finite
simple groups in each of these cases.\end{lemma}

\begin{table} \centering \begin{tabular}{cccc}\hline
&$G$ &$T$ & $|V\Gamma_K|$ \\\hline
1&$A_6$ & $G$ &36\\\hline
2&$M_{12}$ &$G$ or $A_{m}$ & 144\\\hline
3&$\Symp_4(q)$ &$G$ or $A_{m}$ or &$\frac{q^4(q^2-1)^2}{4}$ \\
&($q=2^a>2$)& $\Symp_{4r}(q_0)$ ($q=q_0^r$) &\\ \hline
4&&$\Symp_{4r}(q_0)$ ($q=q_0^r$) &$\frac{q^4(q^2-1)^2}{2}$\\
\hline 5&${\rm P\Omega^+_8}(q)$ &$G$ or $A_{m}$ or
&$\frac{q^6(q^4-1)^2}{(2,q-1)^2}$\\ && $\Symp_8(2)$ (if $q=2$) &\\\hline
\end{tabular}
\caption{Lemma \ref{fpw} (c)(iii).}
\end{table}

A subgroup $K$ of a group $G$ is called {\it core-free} if
$\cap_{g\in G}K^g=1$. Given a core-free subgroup $K$ of $G$ and an
element $g\in G\setminus{\rm N}_G(K)$ such that $g^2\in K$ and
$G=\l K,g\r$, the {\it coset graph} $\Gamma^*=\Gamma(G,K,g)$ is defined by
$$V\Gamma^*=[G:K]=\{\,Kx\mid x\in G\,\},\ \ E\Gamma^*=\{\,\{Kx,Ky\}\mid xy^{-1}\in KgK\,\}.$$
A well known result due to Sabidussi~\cite{Sabidussi} and Lorimer~\cite{Lorimer} asserts that $\Gamma^*$ is $G$-arc-transitive and up to isomorphism every $G$-arc-transitive graph can be constructed this way. The following lemma is a refinement of this result (see \cite[Theorem 2.1]{FPSuzuki}).

\begin{lemma}\label{cosetgraph} Let $\Gamma$ be a finite connected
$(G, 2)$-arc-transitive graph of valency $d$. Then there exists a
core-free subgroup $K$ of $G$ and an element $g\in G$ such that
\begin{enumerate}[\rm (a)]
\item $g\notin{\rm N}_G(K)$, $g^2\in G$, $\l K,g\r=G$;
\item the action of $K$ on $[K:K\cap K^g]$ by right multiplication is transitive, where $|K:K\cap K^g|=d$; and
\item $\Gamma\cong\Gamma(G,K,g)$.
\end{enumerate}
Moreover, one can choose $g$ to be a $2$-element.

Conversely, if $G$ is a finite group with a core-free subgroup $K$
and an element $g$ satisfying (a) and (b) above, then
$\Gamma^*=\Gamma(G,K,g)$ is a connected $(G, 2)$-arc-transitive graph, and $G$
acts faithfully on the vertex set $[G:K]$ of $\Gamma^*$ by right multiplication.
\end{lemma}

\section{Tetravalent 2-transitive non-normal Cayley graphs}

The purpose of this section is to prove the following proposition, which gives all tetravalent 2-transitive non-normal Cayley graphs of finite simple groups. We postpone the definition of $\Gamma(\D_1)$ and $\Gamma(\D_2)$ to (\ref{eqn:3}). 

\begin{proposition}\label{M11} 
Let $G$ be a finite simple group and $\Gamma$ a connected tetravalent 2-transitive non-normal Cayley graph of $G$. Then
$G=\M_{11}$, $\Aut(\Gamma)=\Aut(\M_{12})=\M_{12}{:}2$, $\Aut(\Gamma)_\alpha=S_4$, and $\Gamma$ is isomorphic to $\Gamma(\D_1)$ or $\Gamma(\D_2)$.
\end{proposition}

\begin{proof} 
Suppose that $G$ is a finite simple group and $\Gamma=\Cay(G,S)$ is a connected tetravalent 2-transitive non-normal Cayley graph of $G$. Then, by~\cite[Theorem 1.1]{FLX}, $G$ is one of the following groups:

\begin{equation}\label{eqn:1}
\PSL_2(11),\ \M_{11},\ \M_{23},\ A_{11}.
\end{equation}
Write $A=\Aut(\Gamma)$. Then $A=GA_\alpha$ with $G\cap A_\alpha=1$ and
$A_\alpha=A_4$ or $S_4$ by Lemma~\ref{PGFL}. We consider the following two situations separately.

\bigskip
{\bf Situation 1: $A$ is quasiprimitive on $V\Gamma$.}

In this situation, by Lemma~\ref{fpw} we know that $A$ is almost simple. Let $T=\soc(A)$.
Note that $|S_4|=24$ is divisible by $|A:G|$. It follows that $(T,G)$ is one of the following pairs:
\begin{equation}
\label{eqn:2}
(\M_{11},\ \PSL_2(11)), \ (\M_{12},\ \M_{11}),\
(\M_{24},\ \M_{23}),\ (A_{12},\ A_{11}).
\end{equation}

{\bf Case 1: $(T,G)\in\{(\M_{11},\ \PSL_2(11)),\ (\M_{24},\ \M_{23}),\ (A_{12},\ A_{11})\}$}

First we consider the case $(T,G)=(\M_{11},\ \PSL_2(11))$ and suppose
$\M_{11}=\PSL_2(11)A_4$. It is well known that $M_{11}$ has a faithful
permutation representation of degree 12 acting on $\Omega=\{1,2,\ldots,12\}$. In this
representation, $\PSL_2(11)$ is the point-stabilizer and the subgroup
$A_4$ should be regular on $\Omega$. However, according to the permutation character
$\chi=\chi_1+\chi_{11}$ taken from ATLAS \cite[p. 18]{atlas}, we have $\chi(1A)=12$,
$\chi(2A)=4$ and $\chi(3A)=3$. Therefore, the number of orbits of $A_4$ on $\Omega$ is
$$
\frac{1}{|A_4|}\sum_{g\in A_4}\chi(g)=\frac{1}{12}(12\cdot1+4\cdot3+3\cdot8)=4,
$$
which contradicts the regularity of $A_4$.

Next assume $(T,G)=(\M_{24},\ \M_{23})$. In this case $\M_{24}=\M_{23}K$ for some
subgroup $K\cong S_4$. Since $K$ is regular on $\Omega=\{1,2,\ldots,24\}$, following the
notation of \cite[p. 96]{atlas}, the involution of $K$ must be in class $2B$ and the elements
of order 3 in $3B$. There are two classes of regular elements of order 4, namely $4A$ and $4C$. However,
the power map shows that $4A^2=2A$, which can not be the case. So the elements of order 4 in $K$
must be in $4C$. Now suppose that a 2-element $g\in\M_{24}$ satisfies (a) and (b) in
Lemma~\ref{cosetgraph}. Since $8A^2=4B$, $4A^2=2A$ and $4B^2=2A$, we can only have $g\in 2A$,
$2B$ or $4C$. However, an exhaustive search shows that, for such an element $g$, the subgroup
$\l K,g\r \lneqq \M_{24}$, a contradiction.

Finally, we consider $(T,G)=(A_{12},\ A_{11})$. If $A=A_{12}$, then $A=A_{11}K$ for some subgroup
$K\cong A_4$. Since $K$ is regular on $\Omega=\{1,2,\ldots,12\}$, the involution in $K$ must
be in conjugacy class $2B$ and the elements of order 3 in class $3C$, following the notation of \cite[p. 92]{atlas}. Suppose that a 2-element $g\in A_{12}$ satisfies (a) and (b) in
Lemma~\ref{cosetgraph}. It is evident that $g$ has order $2$ or $4$. According to the power map of conjugacy classes of $A_{12}$, if $g$ has order $4$, then $g^2$ can not be in $2B$. Thus $g$ must have order $2$. Furthermore, $|K\cap K^g|=3$ implies that $g$ normalizes the element of $3C$. With the help of this information, an exhaustive search shows that $\l K,g\r$ can not be $A$. Similarly, when $A=S_{12}$, there is no 2-element satisfying (a) and (b) in Lemma~\ref{cosetgraph}.

The argument above shows that Case 1 does not occur.

\smallskip
{\bf Case 2: $(T,G)=(\M_{12},\ \M_{11})$}

In this case $\Gamma$ is either $(\Aut(\M_{12}),2)$-arc
transitive or $(\M_{12},2)$-arc-transitive. Consider first 
$\Aut(\M_{12})=\M_{12}{:}2$. This group contains a unique class
of subgroups isomorphic to $\M_{11}$. Since $|A:G|=24$, we have
$A_\alpha=S_4$ by Lemma~\ref{PGFL}. Computation using
GAP \cite{Gap} yields the following:

\begin{enumerate}
\item[\rm(a)] $A=\M_{12}{:}2$ has a unique class of subgroups
 $K\cong S_4$ such that $K\cap \M_{11}=1$;

\item[\rm(b)] for a subgroup $K$ in (a), there are in total sixteen
2-elements $g\in A$ such that $K$ and $g$ satisfy (a) and (b) in
Lemma~\ref{cosetgraph}; denote the set of these 16
elements by $\D$;

\item[\rm(c)] ${\rm N}_A(K)=K\cong S_4$ and the conjugate action of
$K$ on $\D$  produces two orbits, denoted by $\D_1$ and $\D_2$,
with $|\D_1|=12$ and $|\D_2|=4$.
\end{enumerate}

Let $K=S_4$ be a subgroup obtained in (a). For any $g$ satisfying (b), the coset
graph $\Gamma(\M_{12}{:}2,S_4,g)$ must be a non-normal
2-transitive tetravalent Cayley graph of $\M_{11}$. Moreover,
for a coset graph $\Gamma(G,K, g)$, it is not difficult to verify that
$\Gamma(G,K,g)\cong\Gamma(G,K^x,g^x)$ for any $x\in \Aut(G)$
(see~\cite[Fact 2.2]{FPSuzuki}). It then follows
that all coset graphs $\Gamma(\M_{12}{:}2\ , K,g)$ with $g\in
\D_i$ are isomorphic, for $i=1,2$. Fix $g_i\in\D_i$ for $i=1,2$. Define

\begin{equation}\label{eqn:3}
\Gamma(\D_i) := \Gamma(\M_{12}{:}2, S_4, g_i),\ i=1,2.
\end{equation}
These two graphs are, up to isomorphism, the only tetravalent
2-transitive non-normal Cayley graphs of $\M_{11}$, for
$\Aut(\Gamma)=\M_{12}{:}2$.

Next we consider $\Gamma(\M_{12},K,g)$. Computation shows that
$\M_{12}$ has a unique class of subgroups $K\cong A_4$
satisfying $K\cap \M_{11}=1$. So we may choose
$K=A_4$ such that $A_4$ is a subgroup of  $S_4$
given in the previous case. In addition, there are in total twelve
2-elements $g$ such that $K$ and $g$ satisfy (a) and (b) in
Lemma~\ref{cosetgraph}. Moreover, these 2-elements are all in
$\D_1$ above and $K$ is transitive on $\D_1$ by
conjugate action. Thus, up to isomorphism, we obtain a unique
tetravalent 2-transitive non-normal Cayley graph of $\M_{11}$,
which is isomorphic to 
$$
\Gamma^*(\D_1) := \Gamma(\M_{12}, A_4, g_1)
$$
for $g_1\in \D_1$.

We claim that $\Gamma^*(\D_1)$ and $\Gamma(\D_1)$ are isomorphic. Note that $A_4$ is contained in
$S_4$ and $\M_{12}$ is transitive on both
$V\Gamma^*(\D_1)$ and $V\Gamma(\D_1)$. Define
$$\sigma:\ \ A_4x\mapsto S_4x,\ \  x\in\M_{12}.$$
It is straightforward to verify that $\sigma$ is an isomorphism from $\Gamma^*(\D_1)$ to $\Gamma(\D_1)$.

Therefore, any quasiprimitive tetravalent 2-transitive non-normal Cayley graph of a finite
simple group is isomorphic to $\Gamma(\D_1)$ or $\Gamma(\D_2)$.

\bigskip
{\bf Situation 2: $A$ is not quasiprimitive on $V\Gamma$.}

In this case, let $H$ be a maximal intransitive normal subgroup of
$A$. Recall that $A=GA_\a$ with $G\cap A_\a=1$, where $A_\alpha\cong A_4$ or $S_4$.
By Lemma~\ref{fpw}, we see that only (c)(i) in Lemma~\ref{fpw}
occurs. This means that $A/H$ is an almost simple group, and $\soc(A/H)$
contains $GH/H\cong G$ and is transitive on $V\Gamma_H$, where
$\Gamma_H$ is the quotient graph of $\Gamma$ relative to $H$. Set 
$T = \soc(A/H)$. 

\smallskip
{\bf Case 1: $T\cong G$}

Since $G$ is simple and $H\lhd A$, we have $H\cap G=1$, which implies that
$|H|$ is a divisor of $|S_4|=24$. If $G$ acts on $H$
nontrivially by conjugation, then $G$ is isomorphic to a subgroup of
$\Aut(H)$. On the other hand, it is not hard to verify that this is
not the case for $G=\PSL_2(11)$, $\M_{11}$, $\M_{23}$
or $A_{11}$. So we assume that $GH=G\times H$. Now $T\cong G$.
It follows that $|\Out(T)|=1$ for $G=\M_{11}$ and $G=\M_{23}$,
while $|\Out(T)|=2$ for $G=\PSL_2(11)$ and $G=A_{11}$. In the former
case we have $A=G\times H$ and hence $G\lhd A$, which is impossible.
In the latter case we have $|A:G\times H|=1$ or 2, which implies
that $G\lhd A$, a contradiction. So Case 1 does not occur.

\smallskip
{\bf Case 2: $T\not\cong G$}

Clearly, $G\cap H=1$ and $|H|$ divides $|A_\alpha|$. So 24 is
divisible by $|H|$. If 3 divides $|H|$, then $|T:GH/H|$ is a divisor
of $8=2^3$, which is impossible by~\cite{Guralnick}. So $H$ is a
$2$-group with $|H|$ dividing 8. Further, if $H_\alpha\ne1$, then
$d(\Gamma_H)=2$, and hence $\Aut(\Gamma_H)$ is a dihedral group, a
contradiction. Hence $H$ is semiregular on $V\Gamma$ and
$d(\Gamma_H)=d(\Gamma)=4$.

For $\a=1\in G=V\Gamma$, set $\bar\alpha=\alpha^H$. Since
$\Gamma$ is $A$-arc-transitive, $\Gamma_H$ is $A/H$-arc
transitive. Moreover, since $(A/H)_{\bar\alpha}=\{Hx\mid x\in
A_\alpha\}$ and $H\cap A_\alpha=1$, $(A/H)_{\bar\alpha}\cong
A_\alpha$. From this it follows that $\Gamma_H$ is $(A/H,2)$-arc
transitive.

Next we determine all pairs $(T,G)$. Note that $|A:G|$ divides $24$. 
So $|A/H:GH/H|$ divides $24/|H|$. Since $H$ is a 2-group,
$24/|H|$ is 6 or 12. Hence, by~\cite{atlas}, $(T,G)$ must be one of
the following pairs:
\begin{equation}\label{eqn:4}
(\M_{11},\PSL_2(11)),\ (\M_{12},\M_{11}),\
(A_{12},A_{11}).\end{equation}
From this we obtain that $|H|=2$ and $\soc(A/H)=A/H=T$. Thus
$\Gamma_H$ is $(T,2)$-arc-transitive with $|V\Gamma_H|=|G|/2$ and
$(T,G)$ given in (\ref{eqn:4}).

Finally, we construct all $(T,2)$-arc-transitive graphs for
$(T,G)$ as given in (\ref{eqn:4}). Note that $|T|/|G|=12$, $|V\Gamma_H|=|G|/2$ and
$|T_{\bar\alpha}|=24$. So $T_{\bar\alpha}\cong S_4$.

Consider $T=\M_{11}$ first. There is only one class of subgroups
isomorphic to $S_4$. Let $K$ be such a subgroup.
Computation using GAP \cite{Gap} shows that there is no
2-element $g$ in $T$ satisfying (a) and (b) in
Lemma~\ref{cosetgraph}, which is a contradiction.

Consider $T=\M_{12}$. Computation shows that there are four classes
of subgroups isomorphic to $S_4$. Using GAP \cite{Gap}, we obtain that there is no 2-element $g\in T$, together with $K\cong S_4$,
satisfying (a) and (b) in Lemma~\ref{cosetgraph}, which is
a contradiction.

Finally, consider $T=A_{12}$. There are 24 conjugate classes of subgroups
$K\cong S_4$. A systematic search using GAP \cite{Gap} shows that there
is no 2-element $g\in T$ such that $K$ and $g$ satisfy (a) and
(b) in Lemma~\ref{cosetgraph}. This completes the proof of
Proposition~\ref{M11}.
\qed
\end{proof}

\section{Proof of Theorem \ref{them1}}

By Proposition~\ref{M11}, a connected tetravalent 2-transitive 
non-normal Cayley graph of a finite simple group is isomorphic to 
$\Gamma(\D_1)$ or $\Gamma(\D_2)$. In this section
we prove that these two graphs are non-isomorphic and thus complete the proof of Theorem \ref{them1}.

\begin{proposition}\label{isomorphism}
The graphs $\Gamma(\D_1)$ and $\Gamma(\D_2)$ defined in (\ref{eqn:3}) are not isomorphic.
\end{proposition}

\begin{proof} 
Write $\Gamma_i=\Gamma(\D_i)$ and $X_i = \Aut(\Gamma)_i$,
for $i=1,2$. It follows from Proposition~\ref{M11} that $\Gamma_1$
and $\Gamma_2$ have the same vertex set and full automorphism group.
Denote $V=V\Gamma_1=V\Gamma_2$ and $X=X_1=X_2=\Aut(\M_{12})=
\M_{12}{:}2$. Now $X_\alpha=S_4$. 
Suppose by way of contradiction that $\Gamma_1\cong \Gamma_2$. Let
$\phi$ be an isomorphism from $\Gamma_1$ to $\Gamma_2$. Then $\phi\in{\rm N}_{\Sym(V)}(X)$
by~\cite[Fact 2.3]{FPSuzuki}. Write $N={\rm N}_{\Sym(V)}(X)$ and
$C={\rm C}_{\Sym(V)}(X)$. Then $N/C$ is isomorphic to a subgroup of
$\Aut(X)$. Moreover, since the vertex stabilizer $X_\alpha\cong 
S_4$ is self-normalized in $X$ (see Case 2, result (c) of Situation 1 in
Section 3), $C=1$ by~\cite[Proposition 2.4]{FPSuzuki} and hence $N$
is a subgroup of $\Aut(X)$. Note that $X=\Aut(\M_{12})$ and
$\Out(X)=1$. Thus $N=X$. It follows that $\phi\in X$ is an automorphism
of $\Gamma_1$, which implies that $\Gamma_1=\Gamma_2$.
On the other hand, for $\alpha=S_4\in V$, the neighbourhood 
$\Gamma_i(\alpha)$ of $\alpha$
in $\Gamma_i$ is given by
$$\Gamma_i(\alpha) = \{S_4g_ix\mid x\in S_4\},\ {\rm for} \ i=1,2.$$
However, computation shows that $\Gamma_1(\alpha)\ne\Gamma_2(\alpha)$, which 
contradicts the statement that $\Gamma_1=\Gamma_2$. Therefore, $\Gamma_1$ and $\Gamma_2$ are not isomorphic.
\qed
\end{proof}

\begin{pf} \textbf{of Theorem \ref{them1}.}
By \cite[Theorem 1.1]{FLX}, Lemma \ref{PGFL}, Proposition \ref{M11} and Proposition \ref{isomorphism}, we obtain Theorem \ref{them1} immediately.
\qed
\end{pf}

\medskip 
\noindent \textbf{Acknowledgements}~~X. G. Fang and J. Wang were supported by the National Natural Science Foundation of China (Grant No. 11931005). S. Zhou was supported by the Research Grant Support Scheme of The University of Melbourne.


\begin{thebibliography}{99}

\bibitem{atlas} J. H. Conway, R. T. Curtis, S. P. Norton, R. A. Parker
and R. A. Wilson, {\it Atlas of Finite Groups}, Clarendon Press,
Oxford, 1985.

\bibitem{FLX} X. G. Fang, C. H. Li and M. Y. Xu, On edge transitive
Cayley graphs of valency four, {\it European J.
Combin.} {\bf25} (2004), 1103--1116.

\bibitem{FPSuzuki} X. G. Fang and C. E. Praeger, Finite two-arc
transitive graphs admitting a Suzuki simple group, {\it Comm. Algebra} {\bf 27} (1999), 3727--3754.

\bibitem{FPW} X. G. Fang, C. E. Praeger and J. Wang, On the
automorphism group of Cayley graphs of finite simple groups, {\it J.
London Math. Soc.} (2) {\bf66} (2002), 563--578.

\bibitem{Gap} The GAP Group, GAP -- Reference Manual, Release 4.7.2, 2013. (http://www.gap-system.org)

\bibitem{Godsil} C. D. Godsil, On the full automorphism group of a graph,
{\it Combinatorica} {\bf1} (1981), 243--256.

\bibitem{GFL} S. T. Guo, Y. Q. Feng and C. H. Li, Edge-primitive
tetravalent graphs, {\it J. Combin. Theory Ser. B}
{\bf112} (2015), 124--137.

\bibitem{Guralnick} R. M.  Guralnick, Subgroups of prime power index in a simple
group, {\it J. Algebra} {\bf 81} (1983), 304--311.

\bibitem{Lorimer} P. Lorimer, Vertex-transitive graphs: Symmetric
graphs of prime valency, {\it J. Graph Theory} {\bf8} (1984), 55--68.

\bibitem{Potocnik} P. Poto\v cnik, A list of 4-valent 2-arc transitive
graphs and finite faithful amalgams of index $(4,2)$, {\it
European J. Combin.} {\bf30} (2009), 1323--1136.

\bibitem{Sabidussi} G. O. Sabiddusi, Vertex-transitive graphs, {\it
Monatsh. Math.} {\bf68} (1964), 426--438.

\bibitem{Xu} M. Y. Xu, Automorphism groups and isomorphisms of Cayley
graphs, {\it Discrete Math.} {\bf 182} (1998), 309--319.

\bibitem{XFWX1} S. J. Xu, X. G. Fang, J. Wang and M. Y. Xu, On cubic
$s$-arc transitive Cayley graphs of finite simple groups, {\it
European J. Combin.} {\bf26} (2005), 133--143.

\bibitem{XFWX2} S. J. Xu, X. G. Fang, J. Wang and M. Y. Xu, $5$-arc
transitive cubic Cayley graphs on finite simple groups, {\it European
J. Combin.} {\bf 28} (2007), 1023--1036.
\end{thebibliography}
\end{document}